\documentclass[11pt,a4paper]{article}
\usepackage[margin=1in]{geometry}
\usepackage{amssymb,amsmath,amsthm}
\usepackage[english]{babel}
\usepackage{microtype}
\usepackage{hyperref}
\hypersetup{colorlinks=true,citecolor=blue,linkcolor=blue,urlcolor=blue}
\usepackage{url}
\usepackage{mathtools}
\usepackage{enumerate}
\usepackage{tikz}
\usepackage{tkz-euclide}
\usetikzlibrary{backgrounds}

\newtheorem{theorem}{Theorem} 
\newtheorem{lemma}[theorem]{Lemma} 
\newtheorem{corollary}[theorem]{Corollary} 
\newtheorem{proposition}[theorem]{Proposition} 
\theoremstyle{definition} 
\newtheorem{remark}[theorem]{Remark} 
\newtheorem{definition}[theorem]{Definition}

\DeclarePairedDelimiter{\parens}{(}{)}
\edef\Erdos{Erd\H{o}s}
\def\Pr{\mathop{{}\rm\bf Pr}\nolimits}
\def\Ev{\mathop{{}\rm\bf E}\nolimits}

\begin{document}

\title{On rainbow-free colourings of uniform hypergraphs\thanks{Stanislav
\v{Z}ivn\'y was supported by a Royal Society University Research Fellowship.
This project has received funding from the European Research Council (ERC) under
the European Union's Horizon 2020 research and innovation programme (grant
agreement No 714532). The paper reflects only the authors' views and not the
views of the ERC or the European Commission. The European Union is not liable
for any use that may be made of the information contained therein.}}

\author{
Ragnar Groot Koerkamp\footnote{This work was done while the first author was at the University
of Oxford.}\\
\texttt{ragnar.grootkoerkamp@gmail.com}
\and
Stanislav \v{Z}ivn\'{y}\\
University of Oxford, UK\\
\texttt{standa.zivny@cs.ox.ac.uk}
}

\date{\today}
\maketitle

\begin{abstract}

  We study rainbow-free colourings of $k$-uniform hypergraphs; that is,
  colourings that use $k$ colours but with the property that no hyperedge
  attains all colours. We show that $p^*=(k-1)(\ln n)/n$ is the threshold
  function for the existence of a rainbow-free colouring in a random $k$-uniform
  hypergraph. 

\end{abstract}

\section{Introduction}
\label{sec:intro}

A $k$-uniform hypergraph $H$ consists of a set of vertices $V(H)$ and a
collection $E(H)$ of $k$-element subsets of $V(H)$, called hyperedges. 
For a $k$-uniform hypergraph $H$, a map $c: V(H) \to [k]$ is called a
\emph{$k$-colouring} of $H$, where $[k]:=\{1,\ldots,k\}$. The colouring $c$ is called \emph{rainbow-free} if for every hyperedge
$e=(v_1,\ldots,v_k)\in E(H)$ we have $c(e)=\{c(v_1),\ldots,c(v_k)\}\neq [k]$ and
for every $i\in [k]$ there is $v\in V(H)$ with $c(v)=i$.

The \emph{$k$-rainbow-free} problem is to determine whether a given $k$-uniform
hypergraph is rainbow-free colourable with $k$ colours.\footnote{The
\emph{$k$-rainbow-free} problem is called \emph{$k$-no-rainbow-colouring}
in~\cite{survey}. For $k=2$, a graph is rainbow-free 2-colourable if and only if
it is disconnected (cf. Remark~\ref{rem:smallk}).}

\paragraph{Contributions}

We initiate the study of $k$-rainbow-free colourings on random hypergraphs. We
consider a natural generalisation of \Erdos-R\'enyi random graphs to random
($k$-uniform) hypergraphs: each possible hyperedge is present with a fixed
probability, independently of the other hyperedges. 
In Section~\ref{sec:threshold}, we find a threshold function for the event that
a random hypergragraph of the first kind is rainbow-free colourable
(Theorem~\ref{thm:threshold}). The proof uses a second moment argument for the
lowerbound and a first moment argument with an analysis of possible types of
rainbow-free colourings for the upperbound. 

\paragraph{Related work}

The $k$-rainbow-free problem is a special case of colouring mixed hypergraphs,
introduced by Voloshin~\cite{mixed_hypergraph_first} and further extended by
Kr\'al', Kratochv\'il, Proskurowski, and Voss~\cite{mixed_hypertrees}. A mixed
hypergraph is a triple $(V,C,D)$ where $V$ is the vertex set and $C$ and $D$ are
collections of subsets of $V$. A colouring of the vertices of a mixed hypergraph
$(V,C,D)$ is called proper if each hyperedge in $C$ contains two vertices of the
same colour and each hyperedge in $D$ contains two vertices of different
colours. The \emph{strict $k$-colouring} problem is to determine whether a given
mixed hypergraph is properly colourable with exactly $k$ colours. The strict
$k$-colouring problem restricted to $k$-uniform mixed hypergraphs with
$D=\emptyset$, so-called co-hypergraphs, is precisely the $k$-rainbow-free
problem. The strict $k$-colouring of co-hypergraphs was later identified, under
the name of \emph{$k$-no-rainbow-colouring}, in the survey by Bodirsky, K\'ara,
and Martin~\cite{survey} as an interesting case of unknown complexity of
\emph{surjective} constraint satisfaction problems on a three-element domain. 

Constraint satisfaction problems (CSPs) are generalisations of graph
homomorphisms~\cite{Hell:graphs}. A graph homomorphism from $G$ to $H$ is a map
from the vertex set of $G$ to the vertex set of $H$ that preserves all edges
(but not necessarily non-edges). For a fixed graph $H$, the
\emph{$H$-coulouring} problems is to determine whether a given graph $G$ admits
a homomorphism to $H$. For instance, taking $H=K_3$ to be the complete graph on
$3$ vertices, $H$-colouring is the well known $3$-colouring problem. 
Hell and Ne\v{s}et\v{r}il established that, unless $H$
contains a loop or is a bipartite graph, the $H$-colouring problem is
NP-complete~\cite{complexity_H_colouring}. 

In an influential paper, Feder and Vardi conjectured that a similar dichotomy
holds for every \emph{digrapgh} $H$, or equivalently, for every finite
relational structure (such as hypergraphs)~\cite{feder_vardi_conjecture}. This
conjecture, known as the CSP dichotomy conjecture, was confirmed by two
independent papers by Bulatov~\cite{Bulatov17:focs} and Zhuk~\cite{Zhuk17:focs},
respectively. While the recent progress on the CSP dichotomy conjecture (and
various CSP variants) relied heavily on the so-called algebraic
approach~\cite{Bulatov05:classifying}, this method does not seem direclty
amenable to \emph{surjective} CSPs, in which we require the homomorphism be
surjective. A dichotomy theorem is known to hold for surjective CSPs on
two-element domains by the work of Creignou and H\'ebrard~\cite{Creignou97}. The
$k$-rainbow-free problem is equivalent to a surjective CSP on a $k$-element domain $[k]$
with a single $k$-ary relation $[k]^k - \{(x_1,\ldots,x_k): x_1,\ldots,x_k
\mbox{ distinct}\}$. Very recently, Zhuk has announced NP-hardness of the
$k$-rainbow-free problem for $k\geq 3$~\cite{Zhuk20:arxiv}.

\section{Preliminaries}
\label{sec:prelim}

If $k$ is clear from the context, we will call a $k$-colouring simply a
colouring. For a colouring $c$ of a $k$-uniform hypergraph, we denote the colour
classes by $C_i:=c^{-1}(i)$, $i\in[k]$.

We state now same basic properties of rainbow-free colourings. 

\begin{definition}\label{def:induced}
  Given a $k$-uniform hypergraph $H$ and a subset of vertices $S\subseteq V(H)$,
  we define the \emph{induced subhypergraph} $H_S$ as the $(k-1)$-uniform
  hypergraph with vertices $V(H_S):=V(H)\setminus S$ and hyperedges $E(H_S) :=
  \{e\cap (V(H)\setminus S) \mid e\in E(H)\text{ and } |e\cap S|=1\}$.

	For up to $k$ disjoint sets $S_1,\dots S_\ell\subseteq V(H)$ we write
    $H_{S_1,\dots,S_\ell}:=((H_{S_1})\dots)_{S_\ell}$ for the repeated induced subhypergraph.
\end{definition}

For $k=3$ in Definition~\ref{def:induced}, $H_S$ will be a graph.
Furthermore, note that the order of the subscripts in the definition of the
repeated induced subhypergraph does not matter.

This notion of induced subhypergraphs is useful because of the following proposition.

\begin{proposition}\label{prop:reduction}
	Let $k\geq 2$ be an integer.
  A $k$-uniform hypergraph $H$ is rainbow-free $k$-colourable if and only if there exists a non-empty subset of vertices
	$\emptyset\neq S\subsetneq V(H)$ such that the $(k-1)$-uniform hypergraph $H_S$
  is rainbow-free $(k-1)$-colourable.
	In particular, this implies the existence of a colouring $c$ of $H$ with $C_k = S$.
\end{proposition}

\begin{proof}
	First suppose that $H$ has a rainbow-free $k$-colouring $c$.
	Let $S$ be $C_k\neq \emptyset$ and consider $H_S$.
	Write $c'$ for the colouring $c$ restricted to $V(H_S) = V(H) \setminus C_k$.
  We will show that $c'$ is indeed a rainbow-free $(k-1)$-colouring of $H_S$.
	First note that $C'_1\cup \dots\cup C'_{k-1} = V(H_S)$,
	and hence every vertex of $H_S$ has a well-defined colour in $[k-1]$.
	Now consider a hyperedge $e'\in E(H_S)$.
	By definition we have $e' = V(H_S) \cap e$ for some $e\in E(H)$.
	We will use a proof by contradiction to show that $c'(e') \neq [k-1]$,
	so assume that $c'(e') = [k-1]$. This implies $[k-1]=c(e') \subseteq c(e)$.
	Furthermore we know that $e\cap S= e\cap C_k\neq \emptyset$, and hence $k\in c(e)$.
	This implies that $c(e) = [k-1]\cup \{k\} = [k]$, which is a contradiction.
	We conclude that $c'(e) \neq [k-1]$ and hence $C'$ is a rainbow-free colouring of $H_S$, as required.

	For the other direction assume that $\emptyset \neq S\subsetneq V(H)$ is such that $H_S$
  has a rainbow-free $(k-1)$-colouring $c'$.
	Now extend $c'$ to a $k$-colouring $c$ of $H$ by setting $c(v) = k$ for all $v\in S$.
	Thus, we have $C_k = S$.
	Let $e$ be a hyperedge in $H$.
	We wish to show that $c(e) \neq [k]$, so that $c$ is a rainbow-free colouring indeed.
	If $|e\cap S|=0$ we have $k\notin c(S)$.
	In the $|e\cap S|=1$ case we have $e' := e\cap (V(H)\setminus S) \in E(H_S)$.
  Since $c'$ is a rainbow-free $(k-1)$-colouring of $H_S$ we know that $c(e') = c'(e') \neq [k-1]$.
	Adding in the one vertex $v$ of $e$ that is in $S=C_k$,
	we get $c(e) = c(e' \cup \{v\}) \neq [k-1]\cup\{k\} = [k]$, as required.
	If $|e\cap S|\geq 2$ there are at most $k-2$ vertices that have a colour in $[k-1]$.
	Since $k-2 < |[k-1]|$ we know that $c(e')$ can not attain all colours in $[k-1]$.
	Hence, this case also implies $c(e)\neq [k]$.
\end{proof}

By induction, it follows that we can apply multiple steps of
Proposition~\ref{prop:reduction} at once.

\begin{corollary}\label{cor:reduction2}
	Let $2\leq \ell<k$ be integers.
	A $k$-uniform hypergraph $H$ is rainbow-free $k$-colourable if and only if
	there exist disjoint non-empty subsets $S_1, \dots, S_\ell$ of $V(H)$ such
	that $H_{S_1,\dots, S_\ell}$ is rainbow-free $(k-\ell)$-colourable.
\end{corollary}

\begin{remark}\label{rem:smallk}
We remark that Proposition~\ref{prop:reduction} also applies to the corner case of $k=2$.
In particular, a graph $H$ is rainbow-free $2$-colourable if and only if there is a subset $S\subsetneq V(H)$ with no outgoing edges; in other words, $H$ is disconnected.
\end{remark}

If all possible rainbow-free hyperedges are
given, not only do we know that the rainbow-free colouring is unique, but we can
also easily find it.

\begin{proposition}
	\label{prop:all_edges}
  Suppose that $H$ is a $k$-uniform hypergraph with a surjective colouring
  $c:V(H)\to[k]$.
	Furthermore assume that $E=\{e\in V^{(k)}\mid c(e)\neq [k]\}$ consists of
  all rainbow-free hyperedges.
	Write $\overline E:=V^{(k)} \setminus E$ for the set of rainbow hyperedges with $c(e)=[k]$.
	The colour classes of $c$ are determined by
    $C_{c(v)}:=\{v\} \cup \{u \in V \mid\forall e\in \overline E: \{u,v\}\nsubseteq e\}$.
\end{proposition}

\begin{proof}
	If $\{u,v\}\subseteq e$ for some $e\in \overline E$,
	we have that $c(u)\neq c(v)$, since $e$ would be a rainbow-free
  hyperedge otherwise.

	For the other direction assume that $c(u)\neq c(v)$.
	By surjectivity of $c$, all colour classes are non-empty and hence there
  exists a vertex $x_j$ for every colour
	$j$ in $[k]-\{c(u),c(v)\}$.
	Using these vertices $x_j$ together with $u$ and $v$ yields a rainbow hyperedge,
	which is an element of $\overline E$.
	Hence there exists a rainbow hyperedge $e\in \overline E$ containing both $u$ and $v$.
  This implies that the condition from the statement of the proposition is both sufficient and necessary.
\end{proof}

\section{Random hypergraphs}
\label{sec:threshold}

The following definition of random hypergraphs is a direct generalisation of the
\Erdos-R\'enyi random graph model: every possible hyperedge is added with a
given probability.

\begin{definition}
	Let $p : \mathbb N \to [0,1]$ be a given probability function.
	A \emph{random $k$-uniform hypergraph} $H^k_{n,p}$ is a $k$-uniform hypergraph created by the following process:
	\begin{itemize}
    \item Start with a set of vertices $V(H^k_{n,p}):=V$ with $|V|=n$.
    \item For each hyperedge $e\in V^{(k)}$, add $e$ to $E(H^k_{n,p})$ with probability $p=p(n)$.
	\end{itemize}
\end{definition}

Let $A$ be a hypergraph property (in our case being rainbow-free colourable). We
write $\Pr[H^k_{n,p}\models A]$ for the probability that $H^k_{n,p}$ satisfies
$A$. A function $r(n)$ is called a \emph{threshold function} for a hypergraph
property $A$ if (i) when $p(n)\ll r(n)$, $\lim_{n\to\infty}\Pr[H^k_{n,p}\models
A]=0$, (ii) when $p(n)\gg r(n)$, $\lim_{n\to\infty}\Pr[H^k_{n,p}\models A]=1$, or
vice versa.

Our main result is the following theorem.

\begin{theorem}
	\label{thm:threshold}
	The function $p^*=(k-1)(\ln n)/n$ is a threshold function for the event that a random $k$-uniform
	hypergraph $H^k_{n,p}$ is rainbow-free colourable.
\end{theorem}

The two parts of the proof, one for small $p$ and one for large $p$, are covered
by the following two lemmas. The result is well known for
$k=2$~\cite[Theorem~VII.9]{Bollobas} and corresponds to disconnectedness (cf.
Remark~\ref{rem:smallk}). Hence we will assume $k\geq 3$.

\begin{lemma}
	\label{lem:small_p}
	For $k\geq 3$, the random hypergraph $H^k_{n,p}$
	is rainbow-free colourable with high probability if $p\leq D\frac{\ln n}n$ for some $D<k-1$.
\end{lemma}
\begin{lemma}
	\label{lem:large_p}
	If $p\geq D(\ln n)/n$ with $D>k-1$ and $k\geq 3$,
	the random hypergraph $H^k_{n,p}$ is not rainbow-free colourable with high probability.
\end{lemma}

  In order to prove Lemma~\ref{lem:small_p}, we use the second moment method;
  i.e, use the second moment of a random variables to bound the probability that
  the variable is far from its mean.  

  Let $X$ be a nonnegative integer-valued random variable such that 
  $X=\sum_{i=1}^m X_i$, where $X_i$ is the indicator variable for event $E_i$.
  For indices $i,j$ write $i\sim j$ if $i\neq j$ and the events $E_i$ and $E_j$
  are \emph{not} independent. We set (the sum is over ordered pairs)
  \[
    \Delta = \sum_{i\sim j} \Pr[E_i \wedge E_j].
  \]

  \begin{proposition}[\protect{\cite[Corollary~4.3.4]{the_prob_method}}]
    \label{prop:alon}
    If $\Ev[X]\to\infty$ and $\Delta=o(\Ev[X]^2)$ then $\Pr[X>0]\to 1$.
  \end{proposition}

\begin{proof}[Proof of Lemma \ref{lem:small_p}]
  Let $H^k_{n,p}$ be a random hypergraph and let $X$ be the number of
  rainbow-free colourings of $H^k_{n,p}$ with only one colour class of size larger than
  one. Our goal is to show that $X>0$ with high probability, and
  thus $H^k_{n,p}$ is rainbow-free colourable with high probability. We will do so by
  invoking Proposition~\ref{prop:alon}.

	We first show that $\Ev[X]$ goes to infinity.

	Let $c$ be a colouring of $H^k_{n,p}$ that uses all $k$ colours and has only one colour class
  of size greater than $1$. 
	We assume that $|C_i|=1$ for $1\leq i\leq k-1$ and $|C_k|=n-k+1$.
	This colouring $c$ is rainbow-free if and only if there are no hyperedges covering all $k$ colour classes.
	There are $1\cdot \dots \cdot 1 \cdot (n-k+1) = n-k+1$ hyperedges with this property, and hence
	\begin{equation*}
		\Pr[c\text{ is a rainbow-free colouring}] = (1-p)^{n-k+1} = \Theta((1-p)^n).
	\end{equation*}
	Since $\ln (1+x) = x + O(x^2)$ for small $x$, we have $1-p = e^{-p + O(p^2)}$ and thus
	\begin{equation*}
		\Pr[c\text{ is a rainbow-free colouring}]
		=\Theta\parens*{e^{-pn + O(p^2n)}}
		=\Theta\parens*{ e^{-D\ln n + O\parens*{D^2(\ln n)^2/n}}}
		= \Theta(n^{-D}).
	\end{equation*}
	The number of colourings $c$ with one large colour class of size $n-k+1$ is $\binom n{n-k+1} =
	\Theta(n^{k-1})$.
	The expected number of such colourings that are rainbow-free is now given by
	\begin{equation*}
		\Ev[X] = \binom n{n-k+1} (1-p)^{n-k+1} = \Theta(n^{k-1} n^{-D}) = \Theta(n^{k-1-D}).
	\end{equation*}
	Since $D<k-1$, this implies that $\Ev[X]\to\infty$ when $n\to\infty$.

  \medskip
  Enumerate all possible colourings $c$ (up to permutations of colours)
	satisfying $|C_k| = n-k+1$ by $c^1$ up to $c^\ell$.
	We write $i\sim j$ if $i\neq j$ and $|C^i_k\cap C^j_k|=n-k$.
	\begin{figure}
		\centering
\begin{tikzpicture}[framed,inner frame xsep=5mm,scale=1.0]
	\tikzset{venn circle/.style={draw,circle,minimum width=5cm}}
	\tikzset{dot/.style={draw,circle,minimum width=1mm,fill=red,color=#1}}

	\node [venn circle] (AM) at (0,0) {};
	\node [venn circle] (BM) at (0,1.5cm) {};
  \node [align=center,above=1.0cm] at (BM) {$A=C^i_k\setminus C^j_k$\\[1mm]$2\leq |A|\leq k$};
  \node [align=center,below=1.0cm] at (AM) {$B=C^j_k\setminus C^i_k$\\[1mm]$2\leq |B|\leq k$};
	\node [align=center] at (0,0.75cm) {$|C^i_k\cap C^j_k| \leq n-k-1$};
  \node [align=right]  at (3.4,-2.3) {$R=V\setminus C^i_k\setminus C^j_k$\\[1mm]$|R|\leq k-3$};
	\node  at (3.6,3.8) {$|V|=n$}; 
\end{tikzpicture}

\caption{This Venn diagram shows the sets $C^i_k$ (the upper circle) and $C^j_k$ (the lower circle) from the proof of Lemma~\ref{lem:small_p}, along with the definitions of $A$, $B$, and $R$. We have $|C^i_k| = |C^j_k| = n-k+1$ and $|C^i_\ell| = |C^j_\ell|=1$ for all $1\leq \ell<k$.} \label{fig:4_venn}
\end{figure}
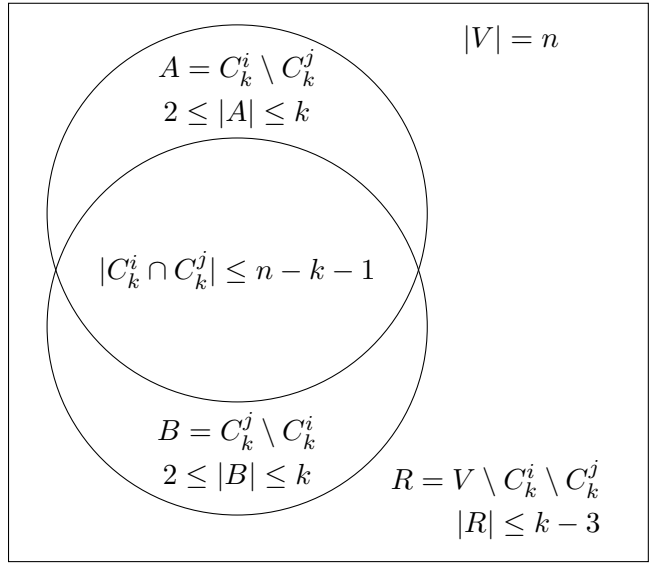
	To every colouring $c^i$ we associate the event $E_i$ that $c^i$ is rainbow-free.

 Consider the quantity
  \begin{equation}\label{def:delta}
    \Delta = \sum_{i\sim j} \Pr[E_i \wedge E_j].
	\end{equation}

	We will prove that $\Delta = o(\Ev[X]^2)$ and 
  and thus finish the proof by Proposition~\ref{prop:alon}. In order for
  Proposition~\ref{prop:alon} to be applicable, we need that (for $i\neq j$) $i\sim j$ if the events $E_i$ and $E_j$ are not independent.

  By the definition of $\sim$, we have
	\begin{equation*}
		\Delta = \sum_{i} \sum_{\substack{j\neq i\\ |C^i_k\cap C^j_k|=n-k}} \Pr[E_i \wedge E_j].
	\end{equation*}

  We claim that the event $E_i$ is independent from $E_j$ if $i\neq j$ and $i\nsim j$.
	In this case, the overlap between $C^i_k$ and $C^j_k$ is at most $n-k-1$,
	since an overlap of $n-k$ implies $i\sim j$ and an overlap of $n-k+1$ implies equality.
	Write $A= C^i_k \setminus C^j_k$, $B=C^j_k \setminus C^i_k$, and
  $R=V(H^k_{n,p}) \setminus C^i_k \setminus C^j_k$,
	as is illustrated in Figure~\ref{fig:4_venn}.
  The colouring $c^i$ is rainbow-free if all hyperedges of the form $e_1 = B\cup R\cup
  \{x\}$ for $x\in C^i_k$
	are not present.
  On the other hand, the colouring $c^j$ is rainbow-free if all hyperedges $e_2=A\cup R\cup\{y\}$ for $y\in C^j_k$ are not
	present.
	We have that $|A| = |C^i_k| - |C^i_k \cap C^j_k| \geq (n-k+1) - (n-k-1) = 2$.
	Similarly we have $|B|\geq 2$.
	Since $A$ is disjoint from $B$, we now know that the hyperedges $e_1$ and $e_2$ can not be equal.
	Hence, the colourings $c^i$ and $c^j$ depend on different hyperedges being present,
	and thus these events are independent indeed.

	Let $i$ and $j$ be such that $i\sim j$; i.e., $i\neq j$ and $|C^i_k\cap C^j_k|=n-k$. In this
  case, we have $|A|=|B|=1$.
  The hyperedges that the events $E_i$ and $E_j$ depend on are of the form $A\cup R\cup \{x\}$ for $x\in C^i_k$
  and $B\cup R\cup \{y\}$ for $y\in C^j_k$ respectively.
	We count $2\cdot(n-k+1)$ hyperedges in total, but the hyperedge $A\cup R\cup B$ is counted twice.
	Hence, the probability that $c^i$ and $c^j$ are both rainbow-free colourings is
	\begin{equation*}
		\Pr[E_i \wedge E_j] = (1-p)^{2(n-k+1)-1} \leq e^{-p(2n-2k+1)} = \Theta(e^{-2pn}).
	\end{equation*}
	Given $c^i$ with $|C^i_k|=n-k+1$, the number of colourings $c^j$ such that the large colour classes
	overlap in $n-k$ positions is $(n-k+1)(k-1)$. Putting this back in $\Delta$ gives
	\begin{align*}
		\Delta
		&= \sum_i (n-k+1)(k-1) (1-p)^{2n-2k+1}\nonumber\\
		&\leq \binom n{n-k+1} (n-k+1)(k-1) e^{-p(2n-2k+1)}\nonumber\\
		&\leq n^{k-1} \cdot n \cdot k \cdot e^{-2D\ln n + O((\ln n)/n)}\nonumber\\
		&=O( n^k\cdot  n^{-2D})
		= O(n^{k-2D}).\label{eq:delta}
	\end{align*}
	Since $k\geq 3$ we have $0 < k-2$ and hence $k-2D < 2k-2-2D = 2(k-1-D)$.
	We conclude that
	\begin{equation*}
		\Delta = O(n^{k-2D}) = o(n^{2(k-1-D)})
	\end{equation*}
	and thus $\Delta = o(\Ev[X]^2)$.
\end{proof}

We will now prove the bound in the other direction, Lemma~\ref{lem:large_p}.
\begin{proof}[Proof of Lemma~\ref{lem:large_p}]
	We use the first moment method to show that the expected number of
  rainbow-free colourings of $H^k_{n,p}$ goes to
	$0$.
	We identify a colouring by the sequence $(s_1,\dots,s_k)$ where $s_i = |C_i|$
  and $s_1\leq \dots \leq s_k$.
	We divide the set of all possible sequences into five types:
	\begin{enumerate}
		\item \label{type1}
			$(s_i)_i = (1,\dots,1,n-k+1)$. There is one such sequence.
		\item \label{type2}
			$(s_i)_i = (1,\dots,1,2,n-k)$. There is one such sequence.
		\item \label{type3}
			$(s_i)_i = (1,\dots,1,x,n-k+2-x)$ with $x\geq 3$. This case contains $O(n)$ sequences.
		\item \label{type4}
			$2\leq s_{k-2}\leq s_{k-1}$ and $s_1+\dots+s_{k-1}\leq 6k$.
			This case contains $O(1)$ sequences, since $k$ is a constant.
		\item \label{type5}
			$2\leq s_{k-2}\leq s_{k-1}$ and $s_1+\dots+s_{k-1}> 6k$.
			This case contains $O(n^{k-1})$ sequences.
	\end{enumerate}
	In each case we will show that the expected number of rainbow-free colourings of the relevant type is $o(1)$, from
	which it follows that the probability that $H^k_{n,p}$ is rainbow-free colourable is $o(1)$.

	Before starting calculations, we introduce some notation.
	We write $\Sigma=  s_1 + \dots + s_{k-1}$ so that $s_k = n - \Sigma \geq n/k$,
	and we write $\Pi = s_1\cdots s_{k-1}$.

	A colouring is rainbow-free if none of the $s_1\cdots s_k$ hyperedges that span all colour classes is present.
	This happens with probability
	\begin{equation*}
		\Pr[c \text{ is rainbow-free}\mid (s_i)_i]
		= (1-p)^{s_1\cdots s_k}
		\leq e^{-ps_1\cdots s_k}
		\leq n^{-D/n \cdot \Pi(n-\Sigma)}.
	\end{equation*}
  Since the number of colourings with a given sequence $(s_i)_i$ is
  upper-bounded by $n^{s_1}\cdots n^{s_{k-1}}=n^{\Sigma}$, the expected number
  of rainbow-free colourings with a given sequence $(s_i)_i$ is bounded by
	\begin{equation}
		\Ev[\text{number of rainbow-free colourings} \mid (s_i)_i]
		\leq n^{\Sigma - D/n \cdot \Pi(n-\Sigma)}.
		\label{eq:ev}
	\end{equation}
  In each of the cases below we will bound the exponent of $n$ in (\ref{eq:ev}).

	Write $D=k-1+\delta$ for some $\delta>0$.

  \textbf{Case~1}: 
	We have $\Sigma = k-1$ and $\Pi = 1$.
  Putting this into~(\ref{eq:ev}) gives an exponent of
  \begin{equation*}
		\Sigma - D/n \cdot \Pi(n-\Sigma) =
		(k-1) - D \cdot (1-(k-1)/n) \to -\delta.
	\end{equation*}
	This is less than $-\delta/2$ if $n$ is large enough.
	Hence, this case is $o(1)$.

  \textbf{Case~2}: 
  Here we have $\Sigma = k$ and $\Pi = 2$. The exponent of $n$ in~(\ref{eq:ev}) becomes
	\begin{equation}
		\Sigma - D/n \cdot \Pi(n-\Sigma) =
		k - (k-1+\delta)\cdot 2 \cdot (1-k/n) \to -k + 2 - 2\delta.
		\label{eq:c2}
	\end{equation}
	Since this converges to a negative number, it will be less than $-1/2$ for all large enough $n$.
	Hence, this case is $o(1)$ as well.

  \textbf{Case~3}: 
  There are $O(n)$ sequences in this case, so each of them must give an expected value that is $o(n^{-1})$.
	The variables are $\Sigma = k+x-2$ and $\Pi = x$.
  The exponent in~(\ref{eq:ev}) is a quadratic function of $x$:
	\begin{equation}
		\Sigma - (k-1+\delta)/n \cdot \Pi(n-\Sigma) =
		k+x-2-(k-1+\delta) \cdot x \cdot (1-(k+x-2)/n).
		\label{eq:quadratic}
	\end{equation}
	Since the leading coefficient is positive, and we want to prove an upper bound, it suffices to check the
	boundaries $x=3$ and $x=n/2$.
	(The maximal possible value of $x$ is actually even smaller, but overestimating doesn't hurt.)
	For $x=3$ we get
	\begin{equation}
		k + 1 - 3(k-1+\delta) (1-(k+1)/n) \to -2k +4 -3\delta < -1.
		\label{eq:c3}
	\end{equation}
	Since this converges to something less than $-1$, we know that the expected value for $x=3$
	is $o(n^{-1})$ for $n$ large enough.

  Since the value of~(\ref{eq:quadratic}) goes to $-\infty$ if $x=n/2$ and
  $n\to\infty$, the upper bound~(\ref{eq:c3}) on the exponent in~(\ref{eq:ev})
  works for the $x=n/2$ case as well.
  
  \textbf{Case~4}: 
	We are given that $\Sigma \leq 6k$. Furthermore we have $s_{k-2}\geq 2$.
	The minimal value of $\Pi$ is attained if $s_1=\dots=s_{k-3}=1$ and $s_{k-1} =
	\Sigma-(k-3)-2=\Sigma-k+1$.
	Thus, we have $\Pi \geq 2(\Sigma - k + 1)$.
	Since $\Sigma$ is a sum of $k-1$ terms, of which the last two are at least $2$, we also have $\Sigma \geq k+1$.
	\begin{align*}
		\Sigma - (k-1+\delta)/n \cdot \Pi(n-\Sigma)
		&\leq \Sigma - (k-1+\delta) 2 (\Sigma -k+1) (1-\Sigma/n)\\
		&\to \Sigma - (k-1+\delta) 2 (\Sigma -k+1).
	\end{align*}
	The step where we take the limit is allowed because $\Sigma$ is bounded, and hence the term divided
	by $n$ goes to $0$ indeed. We continue
	\begin{align}
		\Sigma - (k-1+\delta) 2 (\Sigma -k+1)
		&= \Sigma(1-2(k-1+\delta)) + 2(k-1)(k-1+\delta)\nonumber\\
		&\leq (k+1)(-2k+1-2\delta) + 2(k-1)(k-1+\delta)\nonumber\\
		&= -2k^2-k+1-2k\delta-2\delta + 2k^2-4k+2+2k\delta-2\delta\nonumber\\
		&= -3k+1-4\delta < 0.
		\label{eq:c4}
	\end{align}
	As before this converges to something negative, and hence it will be $o(1)$.

  \textbf{Case~5}: 
	We are now ready for the only remaining case.
	Here we have $\Sigma \geq 6k$ and as before this implies $\Pi \geq 2(\Sigma -k+1)$.
  \begin{equation*}
		\Sigma - (k-1+\delta)/n \cdot \Pi(n-\Sigma)
		\leq
		\Sigma - (k-1+\delta)/n \cdot 2(\Sigma -k+1)(n-\Sigma).
	\end{equation*}
	Using that $s_k = n-\Sigma \geq n/k$ and doing some rewriting gives
	\begin{align*}
		\Sigma - (k-1+\delta)/n \cdot \Pi(n-\Sigma)
		&\leq
		\Sigma - (k-1+\delta)/n \cdot 2(\Sigma -k+1)\frac nk\\
		&=
		\Sigma - \frac{k-1+\delta}k \cdot 2(\Sigma -k+1)\\
		&=
		(1-2(k-1+\delta)/k)\Sigma + 2(k-1)(k-1+\delta)/k\\
		&=
		(-1+2/k-2\delta/k)\Sigma + 2(1-1/k)(k-1+\delta).
	\end{align*}
	We are now at the point where we can use $\Sigma \geq 6k$.
	Because $-1+2/k-2\delta/k < 0$ we get
	\begin{align}
		\Sigma - (k-1+\delta)/n \cdot \Pi(n-\Sigma)
		&\leq (-1+2/k-2\delta/k) \cdot 6k +2(1-1/k)(k-1+\delta)\nonumber\\
		&= -6k + 12-12\delta + 2k-2+2\delta - 2+2/k-2\delta/k\nonumber\\
		&\leq -4k+8-10\delta + 2/k
		\leq -4k+9-10\delta.
		\label{eq:c5}
	\end{align}
	This last value is strictly less than $-k+1$, which is just what we needed.
	We conclude that the total expected number of rainbow-free colourings in this case is $o(1)$ as well,
	and hence the random hypergraph $H^k_{n,p}$ is not rainbow-free colourable with high probability.
\end{proof}

Lemma~\ref{lem:large_p} can be made a bit stronger with respect to the the
colourings of type $(1,\dots,1,n-k+1)$.

\begin{proposition}\label{prop:type1}
	If a random hypergraph $H^k_{n,p}$, with $k\geq 3$, $p = D(\ln n)/n$, and
  $D>k-1$ is rainbow-free colourable then with high probability it has a colouring of type $(1,\dots,1,n-k+1)$.
\end{proposition}
\begin{proof}
	The proof depends heavily on the claims established in the proofs of
  Lemmas~\ref{lem:small_p} and~\ref{lem:large_p}.

  Let $X_i$ be the number of rainbow-free colourings in Case $i$ of the proof of
  Lemma~\ref{lem:large_p}.
	Since $n^{1/n} = e^{(\ln n)/n} \to 1$,
  we know that the convergence of exponents in~(\ref{eq:ev}) in the proof of Lemma~\ref{lem:large_p} implies that
	$n$ raised to the limit of the exponent is off by at most a constant factor. Hence,
	\begin{equation*}
		\mu:=\Ev[X_1] = \Theta(n^{k-1-D}) = \Theta(n^{-\delta}),
	\end{equation*}
	where $D=k-1+\delta$.
	In Cases~2 to~5 of the proof of Lemma~\ref{lem:large_p},
  Equations~(\ref{eq:c2}), (\ref{eq:c3}), (\ref{eq:c4}), and~(\ref{eq:c5}) imply
  that the expected number of rainbow-free colourings in each case
	is bounded by
	\begin{align*}
		\Ev[X_2] &= O(n^{-k+2-2\delta}),\\
    \Ev[X_3] &= O(n)\cdot O(n^{-2k+4-3\delta})= O(n^{-2k+5-3\delta}),\\
		\Ev[X_4] &= O(n^{-3k+1-4\delta}),\\
    \Ev[X_5] &= O(n^{k-1})\cdot O(n^{-4k+9-10\delta}) = O(n^{-3k+8-10\delta}).
	\end{align*}
	Since $k\geq 3$, each of these terms is $o(n^{-1-2\delta})$.
	Hence for $2\leq i\leq 5$ we have
	$ \Pr[X_i>0] \leq \Ev[X_i] = o(n^{-1-2\delta})$.
	To show that almost all random rainbow-free colourable hypergraphs are
  rainbow-free colourable with a colouring of the first
	type indeed, all we have to show is that $\Pr[X_1>0] = \Theta(n^{-\delta})$.

  As in the proof of Lemma~\ref{lem:small_p} enumerate all colourings by $c^1$ to $c^\ell$ and suppose that $c^i$ is a rainbow-free colouring.
	The probability that there is another rainbow-free colouring $c^j$ is bounded by
	\begin{align*}
		\sum_{j\sim i} \Pr[c^j\mid c^i] + \sum_{j\not \sim i,\, j\neq i} \Pr[c^j]
		&\leq n\cdot k \cdot e^{-p(n-k)} + n^{k-1} e^{-p(n-k+1)}\\
		&= O(n^{k-1}n^{-(k-1+\delta)}) = O(n^{-\delta}).
	\end{align*}
	Hence, the probability that the number of rainbow-free colourings is exactly $1$ is at least
	\begin{equation*}
		\sum_{i} \Pr[c^i] (1-O(n^{-\delta})) \sim \sum_i \Pr[c^i] =\Theta(n^{k-1-D}) = \Theta(n^{-\delta}).
	\end{equation*}
  This implies that the probability that $H^k_{n,p}$ is rainbow-free colourable is at
  least $\Theta(n^{-\delta})$. 
\end{proof}

Proposition~\ref{prop:type1} implies that checking colourings of the type
$(1,\dots,1,n-k+1)$ is sufficient to find a colouring in $H^k_{n,p}$ with high
probability \emph{if} we know that the hypergraph is rainbow-free colourable.

\section{Conclusions}

We showed that a threshold function of the event that a random $k$-uniform
hypergraph is rainbow-free colourable is $(k-1)(\ln n)/n$. Our results do not
say anything about the case when the hyperedge probability $p$ is close to the
threshold. As far as we know, the behaviour of the rainbow-free colourings of a
random hypergraph in this case is open.


\end{document}